\documentclass[11pt]{article}
\usepackage{ascmac}
\usepackage{fancybox}
\usepackage{amsmath, amsthm, amscd, amsfonts, amssymb, graphicx,enumerate,color}
  \usepackage{url}

\theoremstyle{definition}
\newtheorem{theorem}{Theorem}[section]

\newtheorem{lemma}[theorem]{Lemma}
\newtheorem{corollary}[theorem]{Corollary}
\newtheorem{proposition}[theorem]{Proposition}
\newtheorem{remark}[theorem]{Remark}

\newcommand{\vertiii}[1]{{\left\vert\kern-0.25ex\left\vert\kern-0.25ex\left\vert #1 
    \right\vert\kern-0.25ex\right\vert\kern-0.25ex\right\vert}}

\begin{document}
\title{Wigner--Yanase--Dyson function and logarithmic mean}
\author{Shigeru Furuichi$^1$\footnote{E-mail:furuichi.shigeru@nihon-u.ac.jp} \\
$^1${\small Department of Information Science,}\\
{\small College of Humanities and Sciences, Nihon University,}\\
{\small 3-25-40, Sakurajyousui, Setagaya-ku, Tokyo, 156-8550, Japan}}
\date{}
\maketitle
{\bf Abstract.} 
The ordering between Wigner--Yanase--Dyson function and logarithmic mean is known. Also bounds for logarithmic mean are known. In this paper, we give two reverse inequalities for Wigner--Yanase--Dyson function and logarithmic mean. We also compare the obtained results with the known bounds of the logarithmic mean. Finally, we give operator inequalities based on the obtained results.
\vspace{3mm}

\vspace{3mm}

\vspace{3mm}


\section{Introduction}
	In this paper, we study the ordering of the symmetric homogeneous means $N(x,y)$ for $x,y>0$.
	The mean $N(x,y)$ is called the symmetric homogeneous mean if the following conditions are satisfied (\cite{HK2003}):
	\begin{itemize}
	\item[(i)] $N(x,y)=N(y,x)$.
	\item[(ii)] $N(kx,ky)=k N(x,y)$ for $k>0$.
	\item[(iii)] $\min\{x,y\}\le N(x,y)\le \max\{x,y\}$.
	\item[(iv)] $N(x,y)$ is non--decreasing in $x$ and $y$.
	\end{itemize}
Since we do not treat the weighted means, a symmetric homogeneous mean is often called a mean simply in this paper.
In order to determine the ordering of two means such as $N_1(x,y)\le N_2(x,y)$ for $x,y>0$, it is sufficient to show the ordering $N_1(x,1)\le N_2(x,1)$ for $x>0$ by homogeneity such that $y N\left(x/y,1\right)=N(x,y)$ for a symmetric homogeneous mean $N(\cdot,\cdot)$ and $x,y>0$. Throughout this paper, we use the standard symbols $A(x,y):=\dfrac{x+y}{2}$, $:=\dfrac{x-y}{\log x-\log y},\,\,(x\neq y>0)$ with $L(x,x):=x$, $G(x,y):=\sqrt{xy}$ and $H(x,y):=\dfrac{2xy}{x+y}$ as the arithmetic mean, logarithmic mean, geometric mean and harmonic mean, respectively. We define the Wigner--Yanase--Dyson function by
\[\left\{ \begin{array}{l}
{W_p}\left( {x,y} \right): =\dfrac{{{{ p\left( {1 - p} \right)\left( {x - y} \right)}^2}}}{{\left( {{x^p} - {y^p}} \right)\left( {{x^{1 - p}} - {y^{1 - p}}} \right)}},\,\,\,\left( x\neq y>0,\,\,\, {p \ne 0,1} \right),\\
{W_p}\left( {x,y} \right): =L(x,y),\,\,\,\left( x\neq y>0,\,\,\, p= 0\,\,\text{or}\,\,1 \right),\\
W_p(x,x):=x,\,\,\,\left(x>0,\,\,\, p\in\mathbb{R} \right).
\end{array} \right.\]
Note that we have $\lim\limits_{p\to 0}W_p(x,y)=\lim\limits_{p\to 1}W_p(x,y)=L(x,y)$.
The Wigner--Yanase--Dyson function was firstly appeared in \cite{PH1996}. 
Since $W_p(x,1)$ is matrix monotone function on $x\in(0,\infty)$ when $-1\le p\le 2$ \cite{Sza2007}, the parameter $p$ is often considered to be $-1\le p \le 2$. We mainly consider the case of $0\le p \le 1$ in this paper, as it was done so in \cite{FY2012,GHI2009,Han2008} to study the  Wigner--Yanase--Dyson metric with Morozova--Chentsov function or the Wigner--Yanase--Dyson skew information. It is easily seen that $W_{1-p}(x,y)=W_{p}(x,y)$ and $W_{1/2}(x,y)=\left(\dfrac{\sqrt{x}+\sqrt{y}}{2}\right)^2$ which is called the Wigner--Yanase function or the power mean (the binomial mean \cite{HK2003}) $B_{p}(x,y):=\left(\dfrac{x^p+y^p}{2}\right)^{1/p}$ with $p=1/2$. It is also known that
$$
H(x,y)\le G(x,y)\le L(x,y) \le W_p(x,y) \le W_{1/2}(x,y) \le A(x,y),\,\,\,(x,y>0,\,\,0\le p\le 1).
$$
	
	The set $M(n,\mathbb{C})$ represents all $n\times n$ matrices on complex field.  The set $M_+(n,\mathbb{C})$ represents all positive semi--definite matrices in $M(n,\mathbb{C})$.
	The stronger ordering $N_1(x,y)\preceq N_2(x,y)$ for means $N_1$ and $N_2$ have been studied in \cite{FA2022,HK1999,HK2003,K2011,K2022} for the study of the unitarily invariant norm inequalities and recent advances on the related topics. 
	
	It is known \cite{HK2003,K2011} that the ordering  $N_1(x,y)\preceq N_2(x,y)$ is equivalent to the  unitarily invariant norm inequality $\vertiii{N_1(S,T)X}\le \vertiii{N_2(S,T)X}$ for $S,T\in M_+(n,\mathbb{C})$ and arbitrary $X\in M(n,\mathbb{C})$, implies the usual ordering $N_1(x,y)\le N_2(x,y)$ which is equivalent to the Hilbert--Schmidt (Frobenius) norm inequality $\| N_1(S,T)X\|_2\le \| N_2(S,T)X\|_2$. See \cite{HK2003,K2011} the precise definition and equivalent conditions on the stronger ordering $N_1(x,y)\preceq N_2(x,y)$. We study the usual ordering for some means in this paper. 
	The following propositions are known.
\begin{proposition}(\cite{HKPR2013})\label{prop1.1}
For $S,T\in M_+(n,\mathbb{C})$ and any $X\in M(n,\mathbb{C})$, 
if $1/2\le p \le 1 \le q \le 2$ or $-1\le q \le 0 \le p \le 1/2$, 
then we have
$$
\vertiii{H(S,T)X} \le \vertiii{W_q(S,T)X} \le \vertiii{L(S,T)X} \le \vertiii{W_p(S,T)X} \le \vertiii{B_{1/2}(S,T)X}.
$$
In particular, $p \in[0,1] \Longrightarrow \vertiii{L(S,T)X} \le \vertiii{W_p(S,T)X}$.
\end{proposition}	
	
\begin{proposition}(\cite{F2014})\label{prop1.2}
	For $S,T\in M_+(n,\mathbb{C})$ and any $X\in M(n,\mathbb{C})$, if $|p|\le 1$, then 
	$$
	\vertiii{\hat{G}_p(S,T)X}\le \vertiii{L(S,T)X} \le \vertiii{\hat{A}_p(S,T)X},
	$$
	where $\hat{G}_p(x,y):=\dfrac{p(xy)^{p/2}(x-y)}{x^p-y^p}$ and $\hat{A}_p(x,y):=\dfrac{p(x^p+y^p)(x-y)}{2(x^p-y^p)}$ for  $x\neq y$.
\end{proposition}

See \cite{F2014} for the details on $\hat{G}_p(x,y)$ and $\hat{A}_p(x,y)$.
From Proposition \ref{prop1.1}, we see $L(x,y) \le W_p(x,y)$ for $0\le p \le 1$. In Section \ref{sec2}, we study the reverse inequalities of $L(x,y) \le W_p(x,y)$. In addition, we compare the obtained results in Section \ref{sec2} with the bounds in Proposition \ref{prop1.2}, in Section \ref{sec3}.
\section{Reverse inequalities}\label{sec2}

For $x>0,t>0$, we have $\ln_{-t} x \le \log x \le \ln_t x$, where $\ln_t x:=\dfrac{x^t-1}{t},\,\,(x>0,\,\,t\neq 0)$. Thus, we have the simple bounds of $W_p,\,\,(0\le p\le 1)$ as
$$
W_p(x,1)\le L(x,1)^2,\,\,\,(x\ge 1),\quad W_p(x,1)\ge L(x,1)^2,\,\,\,(0<x\le 1).
$$

Since $f_p(t):=x^{pt}\log x$ is convex in $t$ when $x\ge 1,\,\,0\le p\le 1$, taking an account for $\int_0^1f_p(t)dt=\ln_px$, we have $x^{p/2}\log x \le \ln_p x \le \left(\dfrac{x^p+1}{2}\right)\log x$ from Hermite--Hadamard inequality. Thus the slightly improved upper bound was obtained under the condition $x \ge 1$:
$$
\frac{4}{(x^p+1)(x^{1-p}+1)}L(x,1)^2\le W_p(x,1)\le \frac{1}{\sqrt{x}}L(x,1)^2,\,\,\,(x\ge 1).
$$
Also,we have the reverse inequality of the above for $0<x\le 1$ since $f_p(t)$ concave in $t$ when $0<x\le 1$.In this section, we study the reverse inequalities of $L(x,y) \le W_p(x,y)$ for {\it all} $x>0$ not restricted as $x\ge 1$ or $0<x\le 1$.

We firstly consider the difference type reverse inequality of $L(x,1)\le W_p(x,1),\,\,(x>0,\,\,0\le p \le 1)$. From the simple calculations, we have
\begin{align}
&W_p(x,1)\le \left(\frac{\sqrt{x}+1}{2}\right)^2\le r\left(\sqrt{x}-1\right)^2+\sqrt{x}\nonumber\\
&\qquad\qquad \le  r\left(\sqrt{x}-1\right)^2+L(x,1),\,\,\,(x>0,\,\,0\le p \le 1,\,\,r\ge 1/4). \label{sec2_eq05}
\end{align}

Considering the parameter $p$, we can obtain the first inequality in the following as a general result.
\begin{theorem}\label{theorem_diff_type}
Let $x>0$. For $0\le p \le 1$, we have
\begin{equation}\label{sec2_eq06}
W_p(x,1)\le  p(1-p)\left(\sqrt{x}-1\right)^2+L(x,1)\le A(x,1).
\end{equation}
\end{theorem}

\begin{proof}
It is trivial that the equalities hold in the inequalities \eqref{sec2_eq06} for the special case $x=1$.
We also  find that the inequalities 
$$ W_p\left(1/x,1\right)\le  p(1-p)\left(\sqrt{1/x}-1\right)^2+ L(1/x,1)\le  A(1/x,1),\quad (x>0)
$$ are equivalent to the inequalities \eqref{sec2_eq06} by multiplying $x>0$ to both sides. Thus it is sufficient to prove the first inequality in \eqref{sec2_eq06} for $x\ge  1$ to show the first inequality in  \eqref{sec2_eq06} for $x>0$.

In the first inequality of \eqref{sec2_eq06}, put $x$ instead of $\sqrt{x}$. Then the denominator is 
$$2(x^{2p}-1)(x^{2(1-p)}-1)\log x\ge 0$$ 
for $x \ge 1$ when we reduce the difference right hand side minus 
the left hand side to a common denominator for the first inequality in \eqref{sec2_eq06}. Then we have the numerator as $(x-1)f(x,p)$, with
$$
f(x,p):=(x+1)(x^{2p}-1)(x^{2(1-p)}-1)-2p(1-p)(x-1)\left(x^p+x^{1-p}\right)^2\log x
$$
Since $f(x,1-p)=f(x,p)$,we have only to prove $f(x,p)\ge 0$ for $x\ge 1$ and $0\le p\le 1/2$.
We calculate
\begin{eqnarray*}
&& \frac{df(x,p)}{dp}=4(x^{1-2p}+1)(\log x)g(x,p),\quad g(x,p):=h(x,p)+p(1-p)(x-1)(x-x^{2p})\log x\\
&& h(x,p):=px^2-(1-p)x^{2p+1}-px+x-px^{2p},\\
&& \dfrac{dh(x,p)}{dx}=2px-(1-p)(1+2p)x^{2p}-2p^2x^{2p-1}+1-p,\\
&& \dfrac{d^2h(x,p)}{dx^2}=2p\left\{-(1-p)(1+2p)x^{2p-1}+p(1-2p)x^{2p-2}+1\right\},\\
&& \dfrac{d^3h(x,p)}{dx^3}=2p(1-p)(1-2p)x^{2p-3}\left\{2p(x-1)+x\right\}\ge 0,\,\,(x \ge 1,\,\,0\le p \le 1/2),
\end{eqnarray*}
so that we have
$$
\dfrac{d^2h(x,p)}{dx^2} \ge \dfrac{d^2h(1,p)}{dx^2}=0\Longrightarrow \dfrac{dh(x,p)}{dx} \ge \dfrac{dh(1,p)}{dx}=0\Longrightarrow h(x,p) \ge h(1,p)=0.
$$
From $p(1-p)(x-1)(x-x^{2p})\log x\ge 0,\,\,(x \ge 1,\,\,0\le p \le 1/2)$ with the above results, we have $\dfrac{df(x,p)}{dp}\ge 0$ which implies $f(x,p)\ge f(x,0)=0$. Thus, we have proved the first inequality in \eqref{sec2_eq06}.

To prove the second inequality in \eqref{sec2_eq06}, we set
$$
k(x,p):=\frac{x+1}{2}-p(1-p)\left(\sqrt{x}-1\right)^2-\frac{x-1}{\log x},\,\,(x >1).
$$
Then we have
$$
k(x,p)\ge k(x,1/2)=\frac{4-4x+(\sqrt{x}+1)^2\log x}{4\log x}\ge 0.
$$
Indeed, we have $\dfrac{x-1}{\log x}\le \left(\dfrac{\sqrt{x}+1}{2}\right)^2$ which implies
$4-4x+(\sqrt{x}+1)^2\log x \ge 0$ for $x>1$. 
This completes the proof with $k(1,p)=0$.
\end{proof}

For the special case $p=1/2$ in Theorem \ref{theorem_diff_type}, the inequalities in \eqref{sec2_eq06} are reduced to $G(x,1)\le L(x,1) \le B_{1/2}(x,1)$.
Note that the right hand side of the second inequality in \eqref{sec2_eq06} can not be replaced by $W_{1/2}(x,1)$ which is less than or equal to $A(x,1)$.

Secondly we consider the ratio type reverse inequality of $L(x,y) \le W_p(x,y)$. From the known results, we have 
\begin{equation}\label{sec2_eq01}
W_p(x,1)\le \left(\frac{\sqrt{x}+1}{2}\right)^2\le A(x,1)\le S(x)G(x,1)\le S(x)L(x,1),\,\,(x>0,\,\,0\le p \le 1).
\end{equation}
Where $S(x):=\dfrac{x^{\frac{1}{x-1}}}{e\log x^{\frac{1}{x-1}}}$ is Specht ratio \cite{S1960}.
From the relation $K(x):=\dfrac{(x+1)^2}{4x}\ge S(x)$, Specht ratio in \eqref{sec2_eq01} can be replaced by Kantorovich constant $K(x)$ \cite{K1948}. See \cite[Chapter 2]{FM2020book} and references therein for the recent results on the inequalities with Specht ratio and Kantorovich constant. 
Moreover we have the following inequality if we use Kantorovich constant $K(x)$.
\begin{equation}\label{sec2_eq02}
W_p(x,1)\le \left(\frac{\sqrt{x}+1}{2}\right)^2= K(\sqrt{x})\sqrt{x}\le K(\sqrt{x})L(x,1),\,\,(x>0,\,\,0\le p \le 1).
\end{equation}

From \eqref{sec2_eq01} and \eqref{sec2_eq02}, it may be expected that $\left(\dfrac{\sqrt{x}+1}{2}\right)^2\le S(\sqrt{x})\sqrt{x}$. However, this fails. Indeed we have the following proposition.
In this point, we see that the ordering $K(x)\ge S(x)$ is effective.
\begin{proposition}
For $x>0$, the following inequality holds: 
\begin{equation}\label{sec2_eq03}
\left(\frac{\sqrt{x}+1}{2}\right)^2\ge S(\sqrt{x})\sqrt{x}.
\end{equation}
\end{proposition}
\begin{proof}
When $x=1$, we have equality of \eqref{sec2_eq03} since $S(1)=1$.
The inequality \eqref{sec2_eq03} is equivalent to the following inequality:
\begin{equation}\label{sec2_eq04}
\frac{(x-1)x^{\frac{x}{x-1}}}{e\log x}\le \left(\frac{x+1}{2}\right)^2.
\end{equation}
By the similar reason as we stated in the beginning of the proof in Theorem \ref{theorem_diff_type},
it is sufficient to prove \eqref{sec2_eq04} for $x>1$. Taking a logarithm of both sides in \eqref{sec2_eq04} and considering its difference:
$$
f(x):=2\log\left(\frac{x+1}{2}\right)-\log(x-1)-\frac{x}{x-1}\log x+1+\log\left(\log x\right).
$$
Since $L(x,1) \ge H(x,1)$ and $L(x,1)^{-1}\ge A(x,1)^{-1}$ for $x>0$,
we have 
$$
f'(x)=\frac{1}{x(x-1)}\left(\frac{x-1}{\log x}+\frac{x\log x}{x-1}-\frac{4x}{x+1}\right)\ge 0,\,\,(x>1).
$$
Thus, we have $f(x)\ge f(1)=0$.
\end{proof}

It is notable that
the inequality \eqref{sec2_eq03} can be also obtaind by putting $v=1/2$ in \cite[Theorem 2.10.1]{FM2020book}, taking a square the both sides and then replacing $x$ by $\sqrt{x}$.

The following result is the ratio type reverse inequality of $L(x,y) \le W_p(x,y)$ for $0\le p \le 1$.
\begin{theorem}\label{theorem2.1}
For $x>0,\,\,0\le p\le 1$, we have
\begin{equation}\label{sec2_eq07}
W_p(x,1)\le K(x)^{p(1-p)} L(x,1).
\end{equation}
\end{theorem}
\begin{proof}
For $x= 1$, we have equality in \eqref{sec2_eq07}. So it is sufficient to prove \eqref{sec2_eq07} for $x>1$. Take a logarithm of the both sides in \eqref{sec2_eq07} and put the function $f(x,p)$ as its difference, namely 
\begin{eqnarray*}
f(x,p)&:=&-\log(x-1)-\log\left(\log x\right)+2p(1-p)\log (x+1)-p(1-p)\log 4x \\
&&-\log p -\log (1-p) +\log(x^p-1)+\log (x^{1-p}-1).
\end{eqnarray*}
We calculate
\begin{eqnarray*}
\frac{df(x,p)}{dx}&=&-\frac{1}{x-1}-\frac{p(1-p)}{x}+\frac{2p(1-p)}{x+1}+\frac{px^{p-1}}{x^p-1}+\frac{p-1}{x^p-x}+\frac{1}{x\log x}\\
&=&-\frac{1}{x-1}+p(1-p)\frac{(x-1)}{x(x+1)}+\frac{1}{x^{1-p}\ln_p x}+\frac{1}{x^p\ln_{1-p}x}+\frac{1}{x\log x}\\
&\ge & -\frac{1}{x-1}+\frac{1}{x^{1-p}\ln_p x}+\frac{1}{x^p\ln_{1-p}x}=:g(x,p)
\end{eqnarray*}
and
$$
g(x,p)=\frac{h(x,p)}{(x-1)(x^p-1)(x-x^p)}\ge 0,\,\,\,(x> 1,\,\,0\le p\le 1).
$$
Indeed, 
\begin{eqnarray*}
h(x,p):&=&-(x^p-1)(x-x^p)+px^{p-1}(x-1)(x-x^p)+(1-p)(x-1)(x-x^p)\\
&=&(1-p)+px-2x^p+(1-p)x^{2p}+px^{2p-1}\\
&=&(1-p)(1+x^{2p})+p(x+x^{2p-1})-2x^p\\
&\ge & 2(1-p) x^p+ 2px^p-2x^p =0.
\end{eqnarray*}
Therefore we have $f(x,p)\ge f(1,p)=0$.
\end{proof}

\begin{remark}
It is natural to consider the replacement $K(x)$ by $S(x)$ in \eqref{sec2_eq07}. However we have not prove
$$
W_p(x,1)\le S(x)^{p(1-p)} L(x,1),\,\,\,\,(x>0,\,\,\,0\le p\le 1).
$$ 
We also have not found any counter-example of the above inequality. 

It is known that $S(x)\le K(x)$, $S(x^r)\le S(x)^r$ and $K(x^r)\le K(x)^r$ for $x>0$ and $0\le r \le 1$ \cite[Section 2.10]{FM2020book}. Also it is known that both $K(x)$ and $S(x)$ are decreasing for $0<x<1$ and increasing $x>1$ with $S(1)=K(1)=1$.

We are interested to find the smaller constant depending $p\in [0,1]$ than $K(x)^{p(1-p)}$.
By the numerical computations we found the counter-example for
$$
W_p(x,1)\le K(x^p) L(x,1),\,\,\,\,(x>0,\,\,\,0\le p\le 1).
$$ 
Thus the inequalities $W_p(x,1)\le S(x^{p}) L(x,1)$ and $W_p(x,1)\le K(x^{p(1-p)}) L(x,1)$ do not hold in general.
\end{remark}

\section{Comparisons of the bounds for logarithmic mean}\label{sec3}
From Proposition \ref{prop1.1} and Theorem \ref{theorem2.1}, we have
\begin{equation}\label{bounds001}
K(x)^{-p(1-p)}W_p(x,1)\le L(x,1)\le W_p(x,1),\,\,\,(x>0,\,\,\,0\le p\le 1).
\end{equation}
On the other hand, from Proposition \ref{prop1.2}, we have
\begin{equation}\label{bounds002}
\hat{G}_p(x,1)\le L(x,1) \le \hat{A}_p(x,1),\,\,\,(x>0,\,\,\,0\le p\le 1).
\end{equation}

In this section, we compare the bounds of $L(x,1)$ in \eqref{bounds001} and  \eqref{bounds002}.
The first result is the comparison on the upper bounds of $L(x,1)$.
\begin{theorem}\label{theorem3.1}
Let $x>0$ and $p\in \mathbb{R}$.
\begin{itemize}
\item[(i)] If $p \in [0, 1/2]$, then $W_p(x,1)\ge \hat{A}_p(x,1)$.
\item[(ii)] If $p \notin (0, 1/2)$, then $W_p(x,1)\le \hat{A}_p(x,1)$.
\end{itemize}
\end{theorem}
\begin{proof}
It is sufficient to prove for $x\ge 1$. Taking  account of $\dfrac{x^{1-p}-1}{1-p}\ge 0$ for $x \ge 1$, we set
$$
f_p(x):=2(x-1)-\frac{(x^p+1)(x^{1-p}-1)}{(1-p)}=\left(2-\frac{1}{1-p}\right)(x-1)+\frac{x^p-x^{1-p}}{1-p}.
$$
Since we have
$$
f_p'(x)=\left(2-\frac{1}{1-p}\right)+\frac{px^{p-1}-(1-p)x^{-p}}{1-p},\,\,\,f_p''(x)=px^{-p-1}(1-x^{2p-1}),
$$
we have $f_p''(x)\ge 0$ for $p \in [0, 1/2]$. Thus, we have $f_p'(x)\ge f_p'(1)=0$ which implies $f_p(x)\ge f_p(1)=0$. Therefore we obtain (i). Similarly we have $f_p''(x)\le 0$ for $p \notin (0, 1/2)$. Thus, we have $f_p'(x)\le f_p'(1)=0$ which implies $f_p(x)\le f_p(1)=0$. Therefore we obtain (ii).
\end{proof} 

The second result is the comparison on the lower bounds of $L(x,1)$.
To prove it, we prepare the following lemma which is interesting itself.
\begin{lemma}\label{lemma3.1}
For $x> 0$, we have
\begin{equation}\label{lemma3.1_00}
\frac{L(x,1)^2-G(x,1)^2}{2L(x,1)^2}\le \frac{A(x,1)-L(x,1)}{L(x,1)}\le \log K(x).
\end{equation}
\end{lemma}
\begin{proof}
It is sufficient to prove \eqref{lemma3.1_00} for $x\ge 1$. We firstly prove the second inequality.
To this end, we set
$$
u(x):=2(x-1)-(x+1)\log x +4(x-1)\log (x+1)-2(x-1)\log (4x),\,\,\,(x\ge 1).
$$
We calculate
\begin{eqnarray*}
&&u'(x)=\frac{4 x}{x+1}-\frac{4}{x+1}+\frac{1}{x}-1 -3 \log x +4 \log (x+1)-4 \log 2\\
&&u''(x)=\frac{(x-1)(x^2+6x+1)}{x^2 (x+1)^2}\ge 0.
\end{eqnarray*}
Thus, we have
$$
u'(x)\ge u'(1)=0\Longrightarrow u(x)\ge u(1)=0.
$$

We secondly prove the first inequality. To this end, we set
$$
v(x):=(x^2-1)\log x+x(\log x)^2-3(x-1)^2,\quad (x\ge 1).
$$
We calculate
\begin{eqnarray*}
&&v'(x)=2(x+1)\log x+(\log x)^2-5x+6-\frac{1}{x},\quad v''(x)=\frac{1}{x^2}\left(2x(x+1)\log x-3x^2+2x+1 \right)\\
&&v^{(3)}(x)=\frac{2}{x^3}w(x),\quad w(x):=x^2-1-x\log x,\quad w'(x)=2x-1-\log x\ge 0.
\end{eqnarray*}
Thus, we have
$$
w(x)\ge w(1)=0\Longrightarrow v^{(3)}(x)\ge 0\Longrightarrow v''(x)\ge v''(1)=0\Longrightarrow v'(x)\ge v'(1)=0\Longrightarrow v(x)\ge v(1)=0.
$$
Therefore we obtain $3L(x,1)\le 2A(x,1)+G(x,1)^2/L(x,1),\,\,\,(x>0)$ which is equivalent to the first inequality of \eqref{lemma3.1_00}.
\end{proof}

Here, the second inequality of \eqref{lemma3.1_00} is equivalent to the inequality:
\begin{equation}\label{lemma3.1_02}
1-\frac{(x+1)\log x}{2(x-1)}+\log \left(\frac{(x+1)^2}{4x}\right)\ge 0
\end{equation}
Also the inequality $\dfrac{L(x,1)^2-G(x,1)^2}{2L(x,1)^2}\le \log K(x)$ is equivalent to the inequality:
\begin{equation}\label{lemma3.1_01}
1-\frac{x (\log x)^2}{(x-1)^2}+2 \log \left(\frac{4 x}{(x+1)^2}\right)\le 0.
\end{equation}

The inequalities \eqref{lemma3.1_02} and \eqref{lemma3.1_01} will be used in the proof of Theorem \ref{theorem3.2} below.

\begin{theorem}\label{theorem3.2}
Let $x>0$ and $p\in \mathbb{R}$.
\begin{itemize}
\item[(i)] If $0\le p \le \dfrac{1}{2}$, then $\hat{G}_p(x,1)\ge K(x)^{-p(1-p)}W_p(x,1)$.
\item[(ii)] If $p\le 0$ or $p\ge 1$, then $\hat{G}_p(x,1)\le K(x)^{-p(1-p)}W_p(x,1)$.
\end{itemize}
\end{theorem}
\begin{proof}
It is sufficient to prove for $x\ge 1$. Since $K(x)\ge 1$, in order to prove (i),
\begin{eqnarray*}
&& \hat{G}_p(x,1)\ge K(x)^{-p(1-p)}W_p(x,1)\Longleftrightarrow \hat{G}_p(x,1)K(x)^{p(1-p)}\ge W_p(x,1) \\
&& \Longleftrightarrow x^{\frac{p}{2}}\left(\frac{(x+1)^2}{4x}\right)^{p(1-p)} \ge \frac{(1-p)(x-1)}{x^{1-p}-1}
\end{eqnarray*}
we set
$$
f_p(x):=\frac{p}{2}\log x+p(1-p)\left\{2\log (x+1)-\log (4x)\right\}-\log(1-p)-\log(x-1)+\log(x^{1-p}-1).
$$
Then we have
\begin{eqnarray*}
f_p'(x)&=&-\frac{1}{x-1}+\frac{p}{2x}+\frac{p(1-p)(x-1)}{x(x+1)}+\frac{1-p}{x-x^p}\\
&=&\frac{g_p(x)}{2(x-1)(x+1)(x-x^p)}
\end{eqnarray*}
where
$$
g_p(x):=2(x+1)(x^p-1-p(x-1))+p(x-1)(1-x^{p-1})\left((3-2p)x+(2p-1)\right).
$$
When $x \ge 1$ and $0\le p \le 1/2$, we have $1\ge x^{p-1}$, $-x^{p-3}\ge -x^{p-2}$ and $(1-p)(1-2p)(p-2)\le 0$.
Also when $x \ge 1$ and $0\le p \le 1/2$, we have $x^{p-1}\ge x^{p-2}$. Thus we calculate
\begin{eqnarray*}
&&g_p'(x)=(p+1)(2p^2-3p+2)x^p-2p(2p^2-2p-1)x^{p-1}+p(1-p)(1-2p)x^{p-2}\\
&&\qquad\qquad +2p(1-2p)x+2(2p^2-2p-1)\\
&&g_p''(x)=p\left\{(p+1)(2p^2-3p+2)x^{p-1}+2(1-p)(2p^2-2p-1)x^{p-2}\right.\\
&&\left.\qquad\qquad+(1-p)(1-2p)(p-2)x^{p-3}+2(1-2p)\right\}\\
&&\ge p\left\{(p+1)(2p^2-3p+2)x^{p-1}+2(1-2p)x^{p-1}+2(1-p)(2p^2-2p-1)x^{p-2}\right.\\
&&\left. \qquad\qquad+(1-p)(1-2p)(p-2)x^{p-2}\right\}\\
&&=(2p^3-p^2-5p+4)(x^{p-1}-x^{p-2})=2(1-p)\left\{(1-p)(1+p)+1-p/2\right\}(x^{p-1}-x^{p-2})\ge 0.
\end{eqnarray*}
Thus, we have $g_p'(x)\ge g_p'(1)=0$ so that $g_p(x)\ge g_p(1)=0$. Therefore we have $f_p'(x)\ge 0$ for $x \ge 1$ and taking an accout for 
$\lim\limits_{x\to 1}\dfrac{(1-p)(x-1)}{x^{1-p}-1}=1$, we have$f_p(x)\ge f_p(1)=0$ which proves (i).

It is also sufficent to prove (ii) for $x\ge 1$. For the special cases $p=0$ or $p=1$ we have equality.
Since 
$$
\hat{G}_p(x,1)\le K(x)^{-p(1-p)}W_p(x,1)\Longleftrightarrow x^{\frac{p}{2}}\left(\frac{(x+1)^2}{4x}\right)^{p(1-p)} \le \frac{(1-p)(x-1)}{x^{1-p}-1},
$$
we have only to prove $f_p(x)\le 0$ for $x\ge 1$.
\begin{itemize}
\item[(a)] We consider the case $p>1$. We set $g_p''(x)=p\cdot h_p(x)$, namely
$$
h_p(x):=(p+1)(2p^2-3p+2)x^{p-1}+2(1-p)(2p^2-2p-1)x^{p-2}+(1-p)(1-2p)(p-2)x^{p-3}+2(1-2p).
$$
Then
$$
h_p'(x)=(p-1)x^{p-4}k_p(x),
$$
where
$$
k_p(x):=2p^3(x-1)^2-p^2(x-1)(x-11)-p(x^2+6x-17)+2(x-3)(x+1)
$$
and we have
$$
k_p'(x)=4p^3(x-1)-2p^2(x-6)-2p(x+3)+4(x-1),\,\,\,k_p''(x)=2(p+1)(2p^2-3p+2)> 0,\,\,(p> 1).
$$
Thus, we have $k_p'(x)\ge k_p'(1)=10p^2-8p>0,\,\,(p> 1)$ so that $k_p(x)\ge k_p(1)=10p-8>0,\,\,(p> 1)$.
Therefore we have $h_p'(x)\ge 0,\,\,(p> 1)$ which implies $h_p(x)\ge h_p(1)=0$. Thus, we have $g_p''(x)\ge 0$ so that we have $g_p'(x)\ge g_p'(1)=0$ which implies $g_p(x)\ge g_p(1)=0$. Taking  account of $x-x^p\le 0$ when $x\ge 1$, $p>1$, we have $f_p'(x)\le 0$ Therefore we have $f_p(x)\le f_p(1)=0$ which proves (ii) for the case $p>1$.
\item[(b)] We consider the case $p<0$. We calculate
\begin{eqnarray*}
&&\frac{df_p(x)}{dp}=\frac{1}{1-p}+\left(\frac{1}{2}+\frac{x}{x^p-x}\right)\log x+(2p-1)\log (4x)+(2-4p)\log (x+1),\\
&&\frac{d^2f_p(x)}{dp^2}=-\frac{x^{p+1} (\log x)^2}{\left(x-x^p\right)^2}+\frac{1}{(p-1)^2}+2 \log (4 x)-4 \log (x+1),\\
&&\frac{d^3f_p(x)}{dp^3}=\frac{x^{p+1} \left(x^p+x\right) (\log x)^3}{\left(x^p-x\right)^3}-\frac{2}{(p-1)^3}.
\end{eqnarray*}
We further calculate
\begin{eqnarray*}
&&\frac{d}{dx}\left(\frac{d^3f_p(x)}{dp^3}\right)=-\frac{x^p(\log x)^2}{(x-x^p)^4}s(x,p),\\
&& s(x,p):=3(x^2-x^{2p})+(p-1)\left(x^2+x^{2p}+4x^{1+p}\right)\log x,\\
&&\frac{ds(x,p)}{dp}=\left\{x^2-5 x^{2 p}+4 x^{p+1}+2 (p-1) \left(x^{2p}+2 x^{p+1}\right)  \log (x)\right\}\log x, \\
&&\frac{d^2s(x,p)}{dp^2}=4x^p(\log x)^2 \left\{2(x-x^p)+(p-1)(x+x^p)\log x\right\}  \le 0 \,\,\,\,(p\le 1,\,\,x\ge 1).
\end{eqnarray*}
Indeed, putting $a:=x$, $b:=x^p$ in the inequality $\dfrac{a-b}{\log a-\log b}\le \dfrac{a+b}{2},\,\,\,(a,b>0)$, we have $2(x-x^p)+(p-1)(x+x^p)\log x\le 0$ for $p<1$ and $x \ge 1$. The equality holds when $p=1$.

From $\dfrac{d^2s(x,p)}{dp^2}\le 0$, we have
$\dfrac{ds(x,p)}{dp}\ge \left.{ \dfrac{ds(x,p)}{dp}} \right|_{p=1}=0$ which implies
$s(x,p)\le s(x,1)=0$ so that we have $\dfrac{d}{dx}\left(\dfrac{d^3f_p(x)}{dp^3}\right)\ge 0$.
From this, we have 
$$
\dfrac{d^3f_p(x)}{dp^3}\ge \dfrac{d^3f_p(1)}{dp^3} =-\frac{2}{(p-1)^3}>0,\,\,\,(p<1).
$$
By the inequality \eqref{lemma3.1_01}, we have
\begin{equation}\label{attention}
p\le 0\Longrightarrow \dfrac{d^2f_p(x)}{dp^2}\le \left.{ \dfrac{d^2f_p(x)}{dp^2} }\right|_{p=0}=1-\frac{x (\log x)^2}{(x-1)^2}+2 \log \left(\frac{4 x}{(x+1)^2}\right)\le 0.
\end{equation}
Thus, we have by the inequality \eqref{lemma3.1_02},
$$
p\le 0\Longrightarrow \frac{df_p(x)}{dp}\ge \left.{\frac{df_p(x)}{dp}}\right|_{p=0}=1-\frac{(x+1)\log x}{2(x-1)}+\log \left(\frac{(x+1)^2}{4x}\right) \ge 0.
$$
Therefore $p\le 0\Longrightarrow f_p(x)\le f_0(x)=0$ which proves (ii) for the case $p<0$.
\end{itemize}

\end{proof}

\begin{remark}\label{remark3.4}
\begin{itemize}
\item[(i)] Since $f_p(x)=\log \left\{\left(\dfrac{(x+1)^2}{4x}\right)^{p(1-p)} \times\dfrac{x^{p/2}(x^{1-p}-1)}{(1-p)(x-1)}\right\}$ appeared in the proof of Theorem \ref{theorem3.2} and comparing the maximum degree of the numerator and the denominator  insides of the logarithmic function, we found that if $p<0$ or $p>1/2$, then we have 
$$\lim\limits_{x\to\infty}\left\{\left(\frac{(x+1)^2}{4x}\right)^{p(1-p)}\times\dfrac{x^{p/2}(x^{1-p}-1)}{(1-p)(x-1)}\right\}=0.$$
Thus, we have $\lim\limits_{x\to\infty}f_p(x) =-\infty$ if $p<0$ or $p>1/2$. On the other hand, by the munerical computations, we have $f_{3/4}(e)\simeq 0.0063209$.
Therefore there is no ordering between $K(x)^{-p(1-p)}W_p(x,1)$ and $\hat{G}_p(x,1)$ for $x>0$ and $1/2< p < 1$. 
\item[(ii)] From Theorem \ref{theorem3.1} (i) and Theorem \ref{theorem3.2} (i), we have for $x>0$ and $0\le p \le 1/2$,
$$
K(x)^{-p(1-p)}W_p(x,1)\le \hat{G}_p(x,1)\le L(x,1) \le \hat{A}_p(x,1) \le W_p(x,1).
$$

\item[(iii)] From the proof (b) in Theorem \ref{theorem3.2}, we obtained 
$\dfrac{d^3f_p(x)}{dp^3}\ge 0,\,\,\,(p<1)$. From this with simple calculations, we have
$$
H(x,x^p)^3\le x^{p+1}A(x,x^p)\le L(x,x^p)^3,\,\,\,(p\le 1,\,\,x> 0).
$$
\end{itemize}
\end{remark}

\section{Conclusion}
As we have seen, we studied the inequalities on the relations between the Wigner--Yanase--Dyson function $W_p(\cdot,\cdot)$ and the logarithmic mean $L(\cdot,\cdot)$. As one of main results, we obtained two kinds of the reverse inequalities for $L(x,y)\le W_p(x,y)$ for $x,y>0$ and $0\le p \le 1$. That is,
the inequalities \eqref{sec2_eq06} and \eqref{sec2_eq07}  shown in Theorem \ref{theorem_diff_type} and \ref{theorem2.1} are respectively equivalent to the following inequalities for $x,y>0$ and $0\le p \le 1$:
\begin{eqnarray}
&&W_p(x,y) \le p(1-p)\left(\sqrt{x}-\sqrt{y}\right)^2+L(x,y),\label{sec4_eq01}\\
&&W_p(x,y) \le K\left(x/y\right)^{p(1-p)} L(x,y).\label{sec4_eq02}
\end{eqnarray}
The inequality \eqref{sec4_eq01} and \eqref{sec4_eq02} are the difference type reverse inequality and the ratio type reverse inequality for $L(x,y)\le W_p(x,y),\,\,(0\le p\le 1)$, respectively.

In addition, we compared the obtained inequality \eqref{sec4_eq02} with the known result in Section \ref{sec3}. It is summerized in the following.
The inequalities given in Remark \ref {remark3.4} (ii)
are equivalent to the following inequalities for $x,y>0$ and $0\le p \le 1/2$:
\begin{equation}
K\left(x/y\right)^{-p(1-p)}W_p(x,y)\le \hat{G}_p(x,y)\le L(x,y) \le \hat{A}_p(x,y) \le W_p(x,y).
\end{equation}

We conclude this paper by giving operator inequalities baesd on Theorem \ref{theorem_diff_type} and \ref{theorem2.1}. For positive operators $S,T$ and $0\le p\le 1$, we define the operator version of the logarithmic mean and the Wigner--Yanase--Dyson function as
\begin{eqnarray*}
&&L(S,T):=\int_0^1 S\sharp_tTdt,\,\,(S\neq T),\quad L(S,S):=S,\\
&&W_p(S,T):=\frac{p(1-p)}{2}(S-T)\left(S\nabla T-Hz_p(S,T)\right)^{-1}(S-T),\,\,(S\neq T,\,\,\,p\neq 0,1),\\
&&W_p(S,T):=L(S,T),\,\, (S\neq T,\,\,\,p= 0\,\,\text{or}\,\,1),\quad W_p(S,S):=S,\,\,(0\le p\le 1),
\end{eqnarray*}
 where
 $$
 S\sharp_pT:=S^{1/2}\left(S^{-1/2}TS^{-1/2}\right)^pS^{1/2},\,\,S\nabla T:=\frac{S+T}{2},\,\,Hz_p(S,T):=\frac{1}{2}\left( S\sharp_pT+ S\sharp_{1-p}T\right).
 $$
 We often use the symbol $S\sharp T:=S\sharp_{1/2}T$ for short. $Hz_p(S,T)$ is often called the Heinz mean. It is notable that we have the relation for the validity in the case $p:=1/2$,
 $$
 \left(\frac{S-T}{2}\right)\left(S\nabla T+S\sharp T\right)^{-1}\left(\frac{S-T}{2}\right)=S\nabla T-S\sharp T,
 $$
 which can be confirmed by multiplying $S^{-1/2}$ to both sides.
 
 From Theorem \ref{theorem_diff_type}, we have the following corollary.
 \begin{corollary}
 Let $S$ and $T$ be positive operators and let $0\le p\le 1$. Then we have
 $$
 W_p(S,T)\le 2p(1-p)\left(S\nabla T-S\sharp T\right)+L(S,T).
 $$
 \end{corollary}
 
 From Theorem \ref{theorem2.1}, we also have the following corollary.
  \begin{corollary}
  Let $S$ and $T$ be positive operators with $\alpha S\le T \le \beta S$ for $0< \alpha \le \beta$ and let $0\le p\le 1$. Then we have
  $$
  W_p(S,T)\le k_p\cdot L(S,T),\quad k_p:=\max_{\alpha \le x\le \beta} K(x)^{p(1-p)}.
  $$
   \end{corollary}
\section*{Acknowledgements}
The authors would like to thank the referees for their careful and insightful comments to improve our manuscript.
The author (S.F.) was partially supported by JSPS KAKENHI Grant Number 21K03341.


\begin{thebibliography}{9}
\bibitem{F2014} S. Furuichi, {\it Unitarily invariant norm inequalities for some means}, J. Inequal. Appl., {\bf 2014}(2014), Art.158.

\bibitem{FM2020book} S. Furuichi and H. R. Moradi, {\it Advances in mathematical inequalities}, De Gruyter, 2020.

\bibitem{FA2022} S. Furuichi and M. E. Amlashi, {\it On bounds of logarithmic mean and mean inequality chain}, arXiv:2203.01134.

\bibitem{FY2012} S. Furuichi and K. Yanagi, {\it Schr\"{o}dinger uncertainty relation, Wigner–Yanase–Dyson skew
information and metric adjusted correlation measure}, J. Math. Anal. Appl., {\bf 388}(2)(2012),  1147--1156.

\bibitem{GHI2009} P. Gibilisco, F. Hansen and T. Isola, {\it On a correspondence between regular and non-regular operator monotone functions}, Linear Algebra Appl., {\bf 430} (2009), 2225--2232.

\bibitem{Han2008} F. Hansen, {\it Metric adjusted skew information}, Proc. Nat. Acad. Sci., {\bf 105}(2008), 9909--9916.

\bibitem{HK1999}F. Hiai and H. Kosaki, {\it Means for matrices and comparison of their norms}, Indiana Univ. Math. J. 48 (1999), 899–936.

\bibitem{HK2003} F. Hiai and H. Kosaki, {\it Means of Hilbert space operators}, Springer--Verlag, 2003.

\bibitem{HKPR2013} F. Hiai, H. Kosaki, D. Petz and B. Ruskai, {\it Families of completely positive maps associated with monotone metrics}, Linear Algebra Appl., {\bf 48}(439)(2013), 1749--1791.

\bibitem{K1948} L. V. Kantorovich, {\it Functional analysis and applied mathematics}, Uspekhi Mat. Nauk, {\bf 3}:6(28) (1948), 89--185. \url{http://mi.mathnet.ru/eng/umn/v3/i6/p89}.

\bibitem{K2011} H. Kosaki, {\it Positive definiteness of functions with applications to operator norm inequalities}, Mem. Amer. Math. Soc., {\bf 212}(997), 2011.

\bibitem{Kosaki2014} H. Kosaki, {\it Strong monotonicity for various means}, J. Func. Anal., {\bf 267}(2014), 1917--1958.

\bibitem{PH1996} D. Petz and H. Hasegawa, {\it On the Riemannian metric of $\alpha$--entropies of density matrices}, Lett. Math. Phys., {\bf 38}(1996), 221-225.

\bibitem{K2022} H. Kosaki, {\it Positive definiteness and infinite divisibility
of certain functions of hyperbolic cosine function}, Internat. J. Math., {\bf 33}(7) (2022), 2250050.

\bibitem{S1960} W. Specht, {\it Zur Theorie der elementaren Mittel}, Math. Z, {\bf 74} (1960), 91--98.
\url{10.1007/BF01180475}.

\bibitem{Sza2007} V. E. S. Szab\'{o}, {\it A class of matrix monotone functions}, Linear Algebra Appl., {\bf 420}(2007), 79--85.

\end{thebibliography}
\end{document}